\documentclass[english]{amsart}
\usepackage{amsmath}

\usepackage{amssymb}
\usepackage{svn}

\usepackage{amscd}
\usepackage[all,cmtip,line]{xy}

\newcommand{\arr}{{\; \rightarrow \;}}

\newcommand{\surjects}{{\; \twoheadrightarrow \;}}
\newcommand{\injects}{{\; \hookrightarrow \;}}

\def\BB{\mathbb{B}}

\def\FF{\mathbb{F}}

\def\HH{\mathbb{H}}

\def\PP{\mathbb{P}}
\def\QQ{\mathbb{Q}}

\def\ZZ{\mathbb{Z}}

\newcommand{\gern}{{\frak{n}}}

\newcommand{\gerp}{{\frak{p}}}
\newcommand{\gerq}{{\frak{q}}}

\newcommand{\gerJ}{{\frak{J}}}

\newcommand{\uA}{{\underline{A}}}
\newcommand{\uB}{{\underline{B}}}

\newcommand{\calF}{{\mathcal{F}}}

\newcommand{\calL}{{\mathcal{L}}}
\newcommand{\calM}{{\mathcal{M}}}

\newcommand{\calO}{{\mathcal{O}}}
\newcommand{\calP}{{\mathcal{P}}}

\newcommand{\ra}{\rightarrow}
\newcommand{\lra}{\longrightarrow}

\newlength{\ownl}

\newcommand{\End}{{\operatorname{End}\,}}

\newcommand{\Gal}{{\operatorname{Gal}\,}}

\newcommand{\Hom}{{\operatorname{Hom}\,}}

\newcommand{\pr}{{\operatorname{pr}\,}}

\newcommand{\Spec}{{\operatorname{Spec}\,}}

\newcommand{\GL}{\operatorname{GL}}

\newcommand{\F}{{\mathbb{F}}}

\newcommand{\Q}{{\mathbb{Q}}}

\newcommand{\gn}{{\mathfrak{n}}}

\newcommand{\Qbar}{{\overline{\Q}}}

\DeclareMathOperator{\lcm}{lcm}

\newcommand{\Qpbar}{\overline{\Q}_p}

\newcommand{\Fpbar}{\overline{\F}_p}

\newcommand{\Fp}{{\F_p}}

\newcommand{\Xbar}{\overline{X}}
\newcommand{\Ybar}{\overline{Y}}

\usepackage{hyperref}

\newcommand{\OF}{\mathcal{O}_F}
\usepackage{amsthm}

\newtheorem*{thmn}{Theorem} 
\newtheorem*{corn}{Corollary} 

\newtheorem{thm}{Theorem}[section]
\newtheorem{corollary}[thm]{Corollary}
\newtheorem{cor}[thm]{Corollary}
 \newtheorem{lemma}[thm]{Lemma}
\newtheorem{lem}[thm]{Lemma} 
\newtheorem{prop}[thm]{Proposition}
 \theoremstyle{definition}
 \theoremstyle{definition}
\newtheorem{defn}[thm]{Definition} \theoremstyle{remark}
\newtheorem{rem}[thm]{Remark}

\numberwithin{equation}{section}

\theoremstyle{definition}

\begin{document}

\title[Weights of Hilbert modular forms]{The cone of minimal weights for mod $p$ Hilbert modular forms}

\author{\sc Fred Diamond}
\email{fred.diamond@kcl.ac.uk}
\address{Department of Mathematics,
King's College London, WC2R 2LS, UK}

\author{\sc Payman L Kassaei}
\email{payman.kassaei@kcl.ac.uk}
\address{Department of Mathematics,
King's College London, WC2R 2LS, UK}

\thanks{F.D.~was partially supported by EPSRC Grant EP/L025302/1.}
\begin{abstract}  
We prove that all mod $p$ Hilbert modular forms arise via multiplication by generalized partial Hasse invariants from
 forms whose weight falls within a certain {\it minimal cone}.  This answers a question posed by Andreatta and Goren,
and generalizes our previous results which treated the case where $p$ is unramified in the totally real field.  Whereas our
previous work made use of deep Jacquet--Langlands type results on the Goren--Oort stratification (not yet available when $p$ is ramified), here we instead use
properties of the stratification at Iwahori level which are more readily generalizable to other Shimura varieties. \end{abstract}

\maketitle



\section{Introduction}   
In this paper, we consider weights of mod $p$ Hilbert modular forms associated to $F$, where 
$F$ is a totally real field of degree $d = [F:\QQ]>1$.    The weights of such forms are $d$-tuples of integers,
and unlike the case of characteristic zero Hilbert modular forms or mod-$p$ classical modular forms,
there are non-zero forms of partially negative weight.  Indeed the partial Hasse invariants studied
by Andreatta and Goren in~\cite{AG} are examples of such forms, and they ask whether all non-zero
forms arise from ones of non-negative weight by multiplication by partial Hasse invariants.
In our earlier work~\cite{DK}, we answered the question in the affirmative, under the assumption that $p$ is
unramified in $F$.  Indeed we proved a stronger result, namely that all non-zero
forms arise from ones whose weight lies in a certain minimal cone contained in the set of
non-negative weights (and is strictly smaller unless $p$ splits completely in $F$).

Here we obtain such a result for arbitrary totally real fields, now allowing $p$ to be ramified in $F$.  
The overall strategy is similar to the one used in \cite{DK}, namely to use $\PP^1$-bundle fibrations
on strata to prove vanishing results, and deduce from these that multiplication by certain partial
Hasse invariants defines an isomorphism for weights outside the minimal cone.
There is however an essential difference which makes the method in this paper more amenable to
generalization: rather than obtaining the desired fibrations from geometric Jacquet--Langlands relations
between Shimura varieties for different reductive groups, we derive the
fibrations from the Hecke action at primes over $p$.

\subsection*{Motivation}  Before describing the main results in more detail, we explain the primary motivation, which stems from the
analogue of Serre's Conjecture in this context.
Recall that Serre's Conjecture~\cite{Serre}, now a theorem of Khare and Wintenberger~\cite{KW1, KW2},
states that every odd, continuous, irreducible representation $\rho:\Gal(\Qbar/\Q) \to \GL_2(\Fpbar)$
arises from a mod $p$ modular form (of level prime to $p$); moreover the minimal level of such a
form is the Artin conductor of $\rho$, and the minimal weight is determined by the restriction of
$\rho$ to $\Gal(\Qbar_p/\Q_p)$.  Analogues of Serre's Conjecture for Galois representations attached
to Hilbert modular forms and more general reductive groups have typically recast the questions in terms
of algebraically-defined Serre weights and the cohomology of arithmetic groups or (usually zero-dimensional)
Shimura varieties (see for example \cite{Ash, BDJ, GHS}).  Such formulations miss potential
geometric input, such as the notion of a minimal weight in the setting of mod $p$ Hilbert modular forms.
This issue is addressed in \cite{DS} (assuming $p$ is unramified in $F$, and in a forthcoming sequel in the ramified case),
where a version of Serre's Conjecture for representations $\rho:\Gal(\overline{F}/F) \to \GL_2(\Fpbar)$ is formulated using
the geometric notion of mod $p$ Hilbert modular forms; in particular the conjecture characterizes a minimal
weight in terms of the restrictions of $\rho$ to $\Gal(\overline{F}_v/F_v)$ for $v|p$.  In addition to predicting
the non-algebraic weights of forms giving rise to $\rho$, the conjectural minimal weight for $\rho$
(and its twists) determines its algebraic Serre weights.  Furthermore considerations from $p$-adic Hodge
theory, modular representation theory and properties of partial $\Theta$-operators led to the prediction
that these minimal weights always lie in a certain non-negative cone, which we call the {\em minimal cone}.

In the classical setting, the non-negativity of minimal weights is immediate from the Riemann--Roch Theorem,
which implies the vanishing of spaces of forms of negative weight.  In the Hilbert modular setting, partial Hasse
invariants associated to non-split primes have partially negative weight, so such a strong vanishing result will
not hold.  On the other hand, the geometric version of Serre's Conjecture in this context predicts that all mod $p$
Hilbert modular forms arise by multiplication by partial Hasse invariants from ones whose weight is not only
non-negative, but lies in the minimal cone, and this is precisely what we prove.  We remark that the geometric
version of Serre's Conjecture further predicts {\em positivity} of minimal weights associated to {\em irreducible}
representations $\rho$; this question is addressed in \cite{DDW} by combining the results of this paper with an argument using partial Frobenius operators (see also \cite{D}).

 Finally, we mention another application of our
results in connection with associated Galois representations:  they are used in the proof of the main theorem
of \cite{DDW}, which establishes unramifiedness at $p$ of Galois representations associated to
certain torsion Hilbert modular forms of parallel weight one, and its generalization in \cite{DM} to certain
cases of partial weight one.

\subsection*{Main results} We now briefly explain the statements of our main results, referring the reader to the body of the paper
for more detailed definitions. 
We fix a sufficiently small level $\gern$ prime to $p$,
and let $\overline{X}^{\rm PR}$ denote the geometric special fibre of the Pappas--Rapoport model (defined in \cite{PR}) 
of the associated Hilbert modular variety (recall $d>1$).
Then $\overline{X}^{\rm PR}_{\Fpbar}$ is smooth of dimension $d$ over $\overline{\FF}_p$, and comes equipped with automorphic line bundles
$\omega_\tau$ indexed by the embeddings $\tau \in \Sigma = \Hom(F, \overline{\QQ}_p)$.  
For ${\bf k} = \sum k_\tau {\bf e}_\tau \in \ZZ^\Sigma$, we consider the space of mod $p$ Hilbert
modular forms of weight $\bf k$, defined as
$$M_{\bf k}(\gern,\overline{\FF}_p) = H^0(\overline{X}^{\rm PR}_{\Fpbar}, \bigotimes_\tau \omega_{\tau}^{\otimes k_\tau}).$$
For each $\tau \in \Sigma$ there is a generalized partial Hasse invariant
$H_\tau$ (defined by Reduzzi and Xiao in \cite{RX} if $p$ is ramified in $F$) of weight
${\bf h}_\tau = n_\tau  {\bf e}_{\sigma^{-1}\tau} - {\bf e}_\tau$,
where the integer $n_\tau \in \{1,p\}$ and $\sigma$ is a certain ``shift'' permutation on $\Sigma$
(generalizing the Frobenius if $p$ is unramified in $F$, in which case $n_\tau = p$ for all $\Sigma$).

Our key theorem then states:
\begin{thmn} If $n_{\tau}k_{\tau} < k_{\sigma^{-1}{\tau}}$ for some for some $\tau\in \Sigma$, then multiplication by $H_{\tau}$
induces an isomorphism:
$$M_{{\bf k}-{\bf h}_{\tau}}(\gn;\Fpbar) \,\,\,\stackrel{\sim}{\lra} \,\,\, M_{\bf k}(\gn; \Fpbar).$$
\end{thmn}
This generalizes the key result of \cite{DK} to the case where $p$ is ramified in $F$.  As mentioned above, the strategy
for proving it is similar; however rather than applying deep results\footnote{As yet unavailable when $p$ is ramified in $F$.}  of~\cite{TX} (refining the geometric Jacquet--Langlands
relation of~\cite{H2}), we give an argument which uses more straightforward properties (proved in \cite{ERX})
on the stratification of the special fibre of the Pappas--Rapoport model of the Hilbert modular variety of Iwahori level
at a prime over $p$ and its behaviour under the degeneracy maps to $\overline{X}^{\rm PR}$.  As in \cite{DK}, the case
where there are split primes over $p$ requires a separate argument, for which we need to establish certain properties of
the Goren--Oort stratification of the Pappas--Rapoport model; these results may be of independent interest.

The theorem allows us to answer the question posed by Andreatta and Goren in~\cite[15.8]{AG} in the ramified case as well
(or more precisely a version which also accounts for the generalized partial Hasse invariants of \cite{RX}).
For a non-zero form $f \in M_{\bf k}(\gn; \Fpbar)$, one can define its {\em filtration} $\Phi(f)$ to be the minimal
weight of a form from which it arises by multiplication by partial Hasse invariants.  
Defining the {\em minimal cone}
$C^{\mathrm{min}} \subset \QQ_{\ge 0}^\Sigma$ to be the set of $\sum x_\tau {\bf e}_\tau$ such that $n_\tau x_\tau \ge x_{\sigma^{-1}\tau}$ for all $\tau$; 
we obtain the corollary:
\begin{corn}  If $0 \neq f \in M_{\bf k}(\gern,\Fpbar)$, then $\Phi(f) \in C^{\mathrm{min}}$.
\end{corn}
As another immediate consequence, we obtain the vanishing
$M_{\bf k}(\gern,\Fpbar)$ for weights outside the Hasse cone, the cone spanned by the ${\bf h}_\tau$
(proved independently in \cite{GKo} in the unramified case).

\subsection*{Acknowledgements}  We are grateful to an anonymous referee for a careful reading of the paper,
and especially for comments leading to improvements of the exposition.


\section{Notation}\label{subsection:notation} Let $F$ be a totally real field of degree $d$ over $\QQ$, and let $\OF$ denote its ring of integers. We assume throughout that $d > 1$. Let $p$ be a prime number,
and $\gern$ an ideal of $\OF$ relatively prime to $p$ which is contained in $N\OF$ for some $N > 3$. Let 
\[
\mathbb{S}:=\{\gerp \in \Spec \calO_F: \gerp|p\}.
\] 
For any $\gerp \in \mathbb{S}$, we let $e_\gerp$ denote the ramification index, and $f_\gerp$ denote the degree of the residue field  $\FF_\gerp:= {\calO_F}/\gerp$ over $\Fp$.  We will be considering a finite field $\FF$ whose degree over $\FF_p$ is divisible by  $f =\lcm\{f_\gerp: \gerp\vert p\}$.

 For every $\gerp\in\mathbb{S}$, we let $F_\gerp$ be the completion of $F$ at $\gerp$,  $\calO_{\gerp}$ its ring of integers, $F_\gerp^{ur}/\QQ_p$ its maximal uramified subextension, and $\calO_{\gerp}^{ur}$ the ring of integers in $F_\gerp^{ur}$. We identify, once and for all, $W(\FF_\gerp)$ with the ring of integers in the maximal unramified sub-extension of $F_\gerp /\QQ_p$, and  fix a  choice of a uniformizer $\varpi_\gerp \in F_\gerp$ for each $\gerp\in\mathbb{S}$. Let 
 \[
 \mathbb{B}:=\textstyle\bigsqcup_{\gerp\in\mathbb{S}}
\mathbb{B}_\mathfrak{p},
\] 
where 
$\mathbb{B}_\mathfrak{p}:={\rm Hom}(F_\gerp^{ur},\Qpbar) \cong {\rm Hom}(W(\FF_\gerp),W(\FF)) \cong {\rm Hom}(\FF_\gerp,\FF)$. Let 
\[
\Sigma:=\textstyle\bigsqcup_{\gerp\in\mathbb{S}}
\Sigma_\mathfrak{p},
\]
where  $\Sigma_\gerp:={\rm Hom}(F_\gerp,\Qpbar)$. Then $\Sigma$ can be identified with ${\rm Hom}(F,\Qpbar)$ via restriction to $F$. There is an obvious  restriction map $res:\Sigma_\gerp \ra \BB_\gerp$ which is $e_\gerp$-to-$1$. For $\beta \in \BB_\gerp$, we denote
\[
res^{-1}(\beta)=\{\tau_{\beta,1},\tau_{\beta,2},\cdots,\tau_{\beta,e_\gerp}\}.
\]
There is no canonical way to order these embeddings, and we choose an ordering as above once and for all. In short, every embedding $\tau:F \arr \Qpbar$ is the restriction to $F$ of a unique embedding of the form $\tau_{\beta,i}$ for some $\beta \in \BB$, and some $1 \leq i \leq e_\gerp$ (where $\gerp$ is such that $\beta \in \BB_\gerp$). We often abuse notation and identify $\tau$ with $\tau_{\beta,i}$. Let $\phi$ denote the automorphism of  $W(\FF)$ that lifts the Frobenius on $\FF$.  We then define $\sigma:\Sigma \to \Sigma$ by $\sigma\tau_{\beta,i} = \tau_{\beta,i+1}$ if $i < e_\gerp$, and $\sigma\tau_{\beta,e_\gerp} = \tau_{\phi\beta,1}$.

 Let $K\subset\Qpbar$ be a finite extension of $\QQ_p$ that contains all the images $\tau_{\beta,i}(F_\gerp)$ for all $\gerp \in \mathbb{S}$, $\beta \in \BB_\gerp$, and $1 \leq i \leq e_\gerp$. Let $\calO$ denote the ring of integers of $K$ (fixing an identification of $W(\FF)$ with a subring of $\calO$). We fix a uniformizer $\varpi$ of $\calO$. Note that we can take $\FF$ to be $\calO/\varpi$.


Let $R$ be an $\calO$-algebra, and $\Lambda$ be a $(\OF \otimes_\mathbb{Z}  R)$-module. The decomposition \[\mathcal{O}_F
\otimes_\mathbb{Z} \calO=\prod_{p|\gerp} \calO_{\gerp} \otimes_{\ZZ_p} \calO=\prod_{\beta \in \mathbb{B}}  \calO_\gerp \otimes_{W(\FF_\gerp),\beta} \calO
\]
induces a decomposition
\begin{eqnarray}\label{equation:decomposition}
\Lambda=\bigoplus_{\gerp \in {\mathbb{S}}} \Lambda_\gerp=\bigoplus_{\beta\in \mathbb{B}} \Lambda_\beta,
\end{eqnarray}
 of $\Lambda$ into $(\calO_F \otimes_\ZZ R)$-modules, where the action of $\calO_F$ on $\Lambda_\gerp$ extends to a continuous  action of $\calO_\gerp$, and  for $\beta \in \BB_\gerp$,  $\Lambda_\beta \subset \Lambda_\gerp$ is the submodule where the action of $W(\FF_\gerp) \subset \calO_\gerp$ is given by  scalar multiplication via $\beta: W(\FF_\gerp) \arr \calO \arr R$. 
 

 \section{Hilbert Modular Varieties: the Pappas-Rapoport model} \label{subsection: HMVtame} In this section, we consider the Pappas-Rapoport integral model of the Shimura variety corresponding to the reductive group $G$ over $\QQ$ defined by  $$G(R)=\{g \in \GL_2(F \otimes_\QQ R)| \det(g) \in R^*\},$$ for a $\QQ$-algebra $R$,  where the level is hyperspecial at $p$. When $p$ is unramified in $F$, this model is the same as the Deligne-Pappas model. When $p$ is ramified in $F$, the Pappas-Rapoport model is  a desingularization of the Deligne-Pappas model.

Let $\gerJ$ be a fractional ideal of $F$, and $\gerJ^+$ its cone of totally positive elements. Let $\gern$ be an ideal of $\calO_F$ prime to $p$ which is contained in $N\calO_F$ for an integer $N>3$.  Let $X_\gerJ^{\rm DP}/\calO$ be the (open) Hilbert modular scheme classifying $\gerJ$-polarized Hilbert-Blumenthal abelian schemes ($\gerJ$-polarized HBAS's)
\[
\underline{A}/S=(A/S,\iota,\lambda,\alpha),
\]
 where, 
 
\begin{itemize}
\item $S$ is a locally noetherian $\calO$-scheme;
\item $A$ is an abelian scheme of relative dimension $d$ over $S$, equipped with
real multiplication $\iota\colon {\calO_F} \rightarrow \End_S(A)$;
\item  $\lambda$ is a polarization as in \cite{DP},
namely, an isomorphism $\lambda\colon (\calP_{A}, \calP_{A}^+)
\rightarrow (\gerJ, \gerJ^+)$ such that $A \otimes_{\calO_F} \gerJ \cong
A^\vee$. Here, $\calP_A = \Hom_{{\calO_F}}(A, A^\vee)^{\rm
sym}$, viewed as sheaf of projective ${\calO_F}$-modules of rank one in the \'etale topology, and $\calP^+_A$ is the cone of polarizations;
\item  the map $\alpha\colon (\calO_F/\gern)^2 \ra A[\gern]$ is an $\calO_F$-linear isomorphism of group schemes. 
\end{itemize}

 We fix  $\{\gerJ_1,\cdots,\gerJ_r\}$  a set of representatives for the strict class group ${\rm Cl}^+(F)$. We define 
 \[
 X^{\rm DP}=\bigsqcup_{i=1}^r X^{\rm DP}_{\gerJ_i}.
 \] 
 The Hilbert modular scheme $X^{\rm DP}$ is a normal scheme of relative dimension $d$ over $\Spec\calO$. Let $\Xbar^{\rm DP}=X^{\rm DP}  \otimes_{\calO} \FF$ be the special fibre, viewed as a closed subscheme of $X^{\rm DP}$.  
 
 Let $S$ be a locally noetherian $\calO$-scheme of connected components $\{S_i\}_{i\in I}$. By a HBAS over $S$ we mean a function $\nu: I \rightarrow \{1,...,r\}$, and a  collection $\{\uA_i\}_{i \in I}$, where $\uA_i$ is a $\gerJ_{\nu(i)}$-polarized HBAS over $S_i$ for each $i \in I$.


Let $\uA/S$ be a HBAS as above, and let $\epsilon_A: A \arr S$ be the structure map. Let  $\omega_{A/S}=\epsilon_{A,\ast} \Omega^1_{A/S}$ and $\HH_{A/S}=R^1\epsilon_{A,\ast} \Omega^\bullet_{A/S};$
these are locally free sheaves of ranks, respectively, $d,2d$ on $S$.  These sheaves carry an action of $\calO_F \otimes \calO_S$, and since $S$ is an $\calO$-scheme, we obtain decompositions as in Equation \ref{equation:decomposition} given by 
\[
\omega_{A/S}=\oplus_{\beta \in \BB}\ \omega_{A/S,\beta},
\]
\[
\HH_{A/S}=\oplus_{\beta \in \BB}\ \HH_{A/S,\beta}.
\] 
For $\beta \in \BB_\gerp$ the sheaves $\omega_{A/S,\beta} \subset \HH_{A/S,\beta}$ are locally free sheaves of rank $e_\gerp$ and $2e_\gerp$, respectively.  We denote the corresponding sheaves on $\Xbar^{\rm DP}$ obtained from the universal HBAS by  $\omega= \oplus_{\beta \in \BB}\ \omega_{\beta}$ and $\HH=\oplus_{\beta \in \BB}\ \HH_{\beta}$.

By a Pappas-Rapoport filtration on $\omega_{A/S}$, we mean a collection of filtrations,
\[
\omega_{A/S}^{\bullet}=(\omega_{A/S,\beta}^\bullet)_{\beta \in \BB},
\]
where, for each $\gerp \in {\mathbb{S}}$ and $\beta \in \BB_\gerp$,  $\omega_{A/S,\beta}^\bullet$ is a filtration of $\omega_{A/S,\beta}$ given by 
\[
0=\omega_{A/S,\beta}^{(0)} \subset ...\subset \omega_{A/S,\beta}^{(i)}\subset ...\subset \omega_{A/S,\beta}^{(e_\gerp)}=\omega_{A/S,\beta},
\]
 satisfying the following conditions:

\begin{enumerate}

\item each $\omega_{A/S,\beta}^{(i)}$ is $\calO_F$-stable;
\item the graded piece $\omega_{A/S,\beta}^{(i)}/\omega_{A/S,\beta}^{(i-1)}$ is a locally free sheaf of rank one over $S$ (in particular, each $\omega_{A/S,\beta}^{(i)}$ is a locally free sheaf of rank $i$ over $S$);
\item the action of $\calO_F$ on $\omega_{A/S,\beta}^{(i)}/\omega_{A/S,\beta}^{(i-1)}$  is the scalar action via $\calO_F \arr \calO_\gerp \overset{\tau_{\beta,i}}{\arr} \calO$. In other words, the action of $\varpi_\gerp$  on $\omega_{A/S,\beta}^{(i)}/\omega_{A/S,\beta}^{(i-1)}$ is given by the scalar action of $\tau_{\beta,i}(\varpi_\gerp) \in \calO$.

\end{enumerate}

There is a scheme $X^{\rm PR}$ over $\Spec\calO$  classifying {\it filtered HBAS's}, 
\[
(A/S,\iota,\lambda,\alpha,\omega^\bullet_{A/S})=(\uA/S,\omega^\bullet_{A/S})
\]
where,

\begin{itemize}
\item $S$ is a locally noetherian $\calO$-scheme;
\item $\uA/S$ is a HBAS  of relative dimension $d$ as above;
\item $\omega^\bullet_{A/S}$ is a Pappas-Rapoport filtration on $\omega_{A/S}$.
\end{itemize}

The (open) Hilbert modular scheme $X^{\rm PR}$ is a scheme smooth of relative dimension $d$ over $\Spec\calO$ (see \cite{PR},\cite{S}).  There is an obvious forgetful morphism 
\[
\delta: X^{\rm PR} \arr X^{\rm DP}.
\]
 It is easy to see that $\delta$ is an isomorphism when restricted to the Rapoport loci of the source and the target, i.e., the loci where the HBAS in question, $\uA/S,$ satisfies the Rapoport condition: $\omega_{A/S}$ is a locally free $(\calO_S \otimes_\ZZ \calO_F)$-module of rank one.
 
We denote by  $\Xbar^{\rm PR}=X^{\rm PR}  \otimes_{\calO} \FF$ the special fibre of $X^{\rm PR}$.  Note that a Pappas-Rapoport filtration over an HBAS $\underline{A}$ over an $\FF$-scheme $S$ is a collection of $\calO_F$-stable filtrations indexed by $\gerp \in \mathbb{S}$, $\beta \in \BB_\gerp$: 
\[
0=\omega_{A/S,\beta}^{(0)} \subset ...\subset \omega_{A/S,\beta}^{(i)}\subset ...\subset \omega_{A/S,\beta}^{(e_\gerp)}=\omega_{A/S,\beta},
\]
where the graded pieces are locally free of rank one, and such that we have 
\[
[\varpi_\gerp] (\omega_{A/S,\beta}^{(i)})\subset \omega_{A/S,\beta}^{(i-1)},
\]
for all $1 \leq i \leq e_\gerp$.
We continue to denote by $\omega, \HH,\omega_\beta, \HH_\beta$ the pullback of these sheaves from $\Xbar^{\rm DP}$ to $\Xbar^{\rm PR}$. For any $(\beta,i)$ as above, we define $\omega_\beta^{(i)}$ to be $\omega_{A/S,\beta}^{(i)}$  where $A/S$ is taken to be the universal filtered HBAS over $\Xbar^{\rm PR}$. This is a locally free sheaf of rank $i$ on $\Xbar^{\rm PR}$. Let $\tau \in \Sigma=\Hom(F,\Qpbar)$. There are unique $\gerp \in \mathbb{S}$ and $\beta \in \BB_\gerp$ such that $\tau$ is the restriction to $F$ of $\tau_{\beta,i}$ for  some $1 \leq  i \leq e_\gerp$. We denote by $\omega_\tau=\omega_{\beta,i}$  the quotient sheaf $\omega_\beta^{(i)}/\omega_\beta^{(i-1)}$. Each $\omega_\tau$ is a locally free sheaf of rank one over $\Xbar^{\rm PR}$.

Let $\gerp \in {\mathbb{S}}$ and $\beta\in \BB_\gerp$. Consider $[\varpi_\gerp]: \HH_\beta \rightarrow \HH_\beta$.  Define, for $1 \leq i \leq e_\gerp$
 \begin{eqnarray}\label{eqn: H_tau}
 \HH_{\beta,i}=([\varpi_\gerp]^{-1}\omega_\beta^{(i-1)})/\omega_\beta^{(i-1)}.
 \end{eqnarray}
 Clearly, we have $\omega_{\beta,i} \subset \HH_{\beta,i}$. If $\tau=\tau_{\beta,i} \in \Sigma$, we set $\HH_\tau=\HH_{\beta,i}$.

  In \cite[\S 2]{RX} a sheaf $\HH^\prime_{\beta,i}$ (denoted there $\HH_{\gerp,j}^{(l)}$ for appropriate choices of $j,l$) is defined as
\[
\HH^\prime_{\beta,i}=\{ z \in   (\omega_{\beta}^{(i-1)})^\perp/ \omega_{\beta}^{(i-1)}\ \vert \ \varpi_\gerp z=0 \},
\]  
where $\perp$ is the orthogonal complement with respect to the pairing  $$\langle.,.\rangle: \HH_\beta \times \HH_\beta \rightarrow \calO_{\Xbar^{\rm PR}}$$ induced by the polarization.
 This pairing satisfies\footnote{See \cite[(2.9.2)]{RX}, but note that the validity of the formula in the case $p = 2$ can be deduced immediately from the flatness of $X^{\rm PR}$ over $\calO$ instead of the assertions in the paragraph following {\it loc. cit.}.  See also \cite[\S3.1]{D} for an analysis of the pairing and an analogue of Lemma~\ref{lem:orthogonal} over the integral model $X^{\rm PR}$.} 
 \begin{equation}\label{eqn: pairing} 
 \langle ax,x \rangle=\langle x,ax \rangle = 0
 \end{equation}
 for $x \in \HH_\beta, a\in \calO_\gerp$. It is proven in \cite[Cor. 2.10]{RX} that $\HH^\prime_{\beta,i}$ is a locally free sheaf of rank $2$, and there is an isomorphism $\wedge^2 \HH^\prime_{\beta,i} \rightarrow \calO_{\Xbar^{\rm PR}}$. 
 
 \begin{lem}\label{lem:orthogonal} We have $\HH_{\beta,i} =\HH^\prime_{\beta,i}$. In particular, $\HH_{\beta,i}$ is locally free of rank $2$ on $\Xbar^{\rm PR}$, and we have a (non-canonical) isomorphism $\wedge^2 \HH_{\beta,i} \cong \calO_{\Xbar^{\rm PR}}$.                     
 \end{lem}
 
 \begin{proof} Clearly, we have  $\HH_{\beta,i}  \supset \HH^\prime_{\beta,i}$. Since $\HH$ is locally free of rank $2$ over $\calO_{\Xbar^{\rm PR}} \otimes_{\ZZ} \calO_F$, we know for all $\beta \in \BB_\gerp$,  $\HH_\beta$ is locally free of rank $2$ over $\calO_{\Xbar^{\rm PR}} \otimes_{W(\FF_\gerp),\beta} \calO_\gerp\cong \calO_{\Xbar^{\rm PR}}[\![u]\!]/\langle u^{e_\gerp}\rangle$, where the isomorphism is chosen so that $\varpi_\gerp$ is sent to $u$. Let $x$ be a closed point of $\Xbar^{\rm PR}$ with residue field $\kappa$.  Hence $(\HH_\beta)_x \cong (\kappa[[u]]/\langle u^{e_\gerp}\rangle)^{\oplus  2}$, where a basis is chosen such that the submodule $\omega_{\beta}^{(i-1)}$ is given by  $\frac{u^a  \kappa[\![u]\!]}{\langle u^{e_\gerp}\rangle } \oplus  \frac{u^ b\kappa[\![u]\!]}{\langle u^{e_\gerp}\rangle }$ with $a,b$ positive integers satisfying $a+b=2e-i+1$. It follows that 
 \[
 (\HH_{\beta,i})_x = [\varpi_\gerp]^{-1}(\omega_\beta^{(i-1)})_x/(\omega_\beta^{(i-1)})_x \cong \frac{u^{a-1} \kappa[\![u]\!]/\langle u^{e_\gerp}\rangle }{u^a \kappa[\![u]\!]/\langle u^{e_\gerp}\rangle } \oplus  \frac{u^{b-1} \kappa[\![u]\!]/\langle u^{e_\gerp}\rangle}{u^b \kappa[\![u]\!]/\langle u^{e_\gerp}\rangle} 
 \]
 which has dimension $2$.  On the other hand, since $a+b>e$,  Equation \ref{eqn: pairing} implies that $ (\HH_{\beta,i})_x \subset  (\omega_{\beta}^{(i-1)})_x^\perp/ (\omega_{\beta}^{(i-1)})_x$. This ends the proof.
\end{proof}

  If $i>1$, the action of $\varpi_\gerp$ induces a map $$[\varpi_\gerp]:\HH_{\beta,i} \rightarrow \HH_{\beta,i-1}$$
 which clearly has image $\omega_{\beta,i-1}$.

\begin{defn} Let $R$ be an $\FF$-algebra. Let ${\bf k}=\sum k_\tau  {\bf e}_\tau \in \ZZ^{\Sigma}$. We define
\[
\omega^{\bf k}:=\bigotimes_\tau \omega_{\tau}^{\otimes k_\tau}
\]
The space of mod $p$ Hilbert modular forms of weight $\bf k$ and level $\Gamma(\gn)$ over $R$ is defined to be
\[
M_{\bf k}(\gn;R)=H^0(\Xbar^{PR} \otimes_{\FF} R, \omega^{\bf k}).
\]
\end{defn}
 Note that under our assumption that $d>1$, the Koecher principle implies that this is the same as the space of sections of a suitable extension of $\omega^{\bf k}$
 over a toroidal compactification of $\Xbar^{PR}$.
 
\section{The generalized partial Hasse invariants} \label{section: Hasse} We recall the definition of generalized partial Hasse invariants introduced in \cite{RX} which factorize the partial Hasse invariants defined by \cite{AG}.  Let $\gerp \in \mathbb{S}$ and $\beta \in \BB_\gerp$. The Verschiebung morphism of the universal filtered HBAS over  the Hilbert modular variety $\Xbar^{{\rm PR}}$ 
 \[
{\rm Ver}: A^{(\phi)} \rightarrow A 
\]
induces ${\rm Ver}^*_\beta: \HH_\beta \surjects \omega_{\phi^{-1}\beta}^{(\phi)} \subset  \HH_{\phi^{-1}\beta}^{(\phi)}$ (where $\cdot^{(\phi)}$ denotes pullback by absolute Frobenius on $\Xbar^{{\rm PR}}$).  The restricted map ${\rm Ver}^*_\beta: \omega_{\beta} \rightarrow \omega_{\phi^{-1}\beta}^{(\phi)}$ preserves the universal Pappas-Rapoport filtration and hence induces
\[
{\rm Ver}^*_\beta: \omega_{\beta,e_\gerp}=\omega_\beta/\omega_\beta^{(e_\gerp-1)} \rightarrow     \omega_{\phi^{-1}\beta,e_\gerp}^{\otimes p}=\omega_{\phi^{-1}\beta}^{(\phi)}/(\omega_{\phi^{-1}\beta}^{(e_\gerp-1)})^{(\phi)}.
\]
We let $H_\beta \in H^0(\Xbar^{\rm PR}, \omega_{\phi^{-1}\beta,e_\gerp}^{\otimes p} \otimes  \omega_{\beta,e_\gerp}^{-1})$ denote the section defined by the above morphism.
It is easy to see that on the Rapoport locus of $\Xbar^{\rm PR}$, $H_\beta$ corresponds under $\delta$ to a partial Hasse invariant (at the embedding $\beta$) defined by \cite{AG}  on the Rapoport locus of $\Xbar^{\rm DP}$. In \cite{RX}, ${\rm Ver}^*_\beta$ is factorized as follows. For $i>1$, let 
\[
m_{\beta,i}: \omega_{\beta,i}=\omega_\beta^{(i)}/\omega_\beta^{(i-1)} \rightarrow
\omega_{\beta,i-1}=\omega_\beta^{(i-1)}/\omega_\beta^{(i-2)}
\]
be the morphism defined by the action of $\varpi_\gerp \in \calO_\gerp$. Define a morphism $${\rm V}_{\beta,1}: \HH_{\beta,1} \rightarrow \omega_{\phi^{-1}\beta,e_\gerp}^{\otimes p}$$  as follows: Let $h$ be a local section of $\HH_{\beta,1}$. Since $[\varpi_\gerp]h=0$, it follows that $h=[\varpi_\gerp]^{e_\gerp-1}h'$ for $h'$ some local section of $\HH_\beta$. Define ${\rm V}_{\beta,1}(h)$ to be the image of  ${\rm Ver}_\beta^*(h') \in \omega_{\phi^{-1}\beta}^{(\phi)}$ in $\omega_{\phi^{-1}\beta,e_\gerp}^{\otimes p} $. It  follows that we have the factorization
\[
{\rm Ver}^*_\beta=
{{\rm V}_{\beta,1}}|_{\omega_{\beta,1}} \circ m_{\beta,2}  \circ \cdots  \circ m_{\beta,e_\gerp}.
\]

Let ${H}_{\beta,1} \in H^0(\Xbar^{\rm PR}, \omega_{\phi^{-1}\beta,e_\gerp}^{\otimes p} \otimes  \omega_{\beta,1}^{-1})$ be the section defined by the restriction of ${\rm V}_{\beta,1}$ to $\omega_{\beta,1}$.
Also, for $1<i\leq e_\gerp$, let ${H}_{\beta,i} \in H^0(\Xbar^{\rm PR}, \omega_{\beta,i}^{-1} \otimes  \omega_{\beta,i-1})$ be the section defined by $m_{\beta,i}$. It follows that the pullback under $\delta$ of the partial Hasse invariant (at $\beta$) defined in \cite{AG}  on the Rapoport locus of $\Xbar^{\rm DP}$ extends from the Rapoport locus of $\Xbar^{\rm PR}$ to $\Xbar^{\rm PR}$, and factorizes as ${H}_\beta = \Pi_{i=1}^{e_\gerp} {H}_{\beta,i}$.  

Let $\tau \in \Sigma=\Hom(F,\Qpbar)$, so that $\tau=\tau_{\beta,i}$ for a unique $\gerp \in \mathbb{S}$, $\beta \in \BB_\gerp$, and   $1 \leq i \leq e_\gerp$.  We define ${H}_\tau=H_{\beta,i}$. Then $H_\tau$ is a mod $p$ Hilbert modular form of weight ${\bf h}_\tau$ which equals  ${\bf e}_{\sigma^{-1}\tau}-{\bf e}_\tau$ if $i>1$, and $p{\bf e}_{\sigma^{-1}\tau}-{\bf e}_\tau$ if $i=1$.



\section{The Goren-Oort stratification on $\Xbar^{{\rm PR}}$}

The partial Hasse invariants ${H_\tau}$ have been used by Reduzzi and Xiao to define an analogue of the Goren-Oort stratification on $\Xbar^{{\rm PR}}$ (see \cite{RX}).  For any $\tau \in \Sigma=\Hom(F,\Qpbar)$,  define $Z_{\tau}$ to be the closed subscheme of $\Xbar^{{\rm PR}}$ given by the vanishing of ${H_\tau}$. For $T \subset \Sigma$, define a closed subscheme of $\Xbar^{{\rm PR}}$
\[
Z_T=\bigcap_{\tau \in T} Z_\tau,
\]
as well as  an open subscheme of $Z_T$ given by 
\[
W_T=Z_T -\bigcup_{\tau \not \in T} Z_\tau.
\]
Reduzzi and Xiao show \cite[Thm. 3.10]{RX} that if $T \neq \emptyset$, then $Z_T$ is a smooth proper subscheme of $\Xbar^{\rm PR}$ (see also \cite[Prop. 4.2.1]{D}). In the following, we prove some properties of this stratification that we use in the proof of our main theorem.

\begin{lem} Let $\gerp \in \mathbb{S}$,  $\beta \in \BB_\gerp$, and $1 \leq i \leq e_\gerp$. Let $T \subset \Sigma$ and $\tau_{\beta,i} \in T$. Then, if $i>1$, we have 
\[
\omega_{\beta,i}|_{Z_T} \cong \omega_{\beta,i-1}^{-1}|_{Z_T}.
\]
For $i=1$, we have 
\[
\omega_{\beta,1}|_{Z_T} \cong \omega_{\phi^{-1}\beta,e_\gerp}^{-p}|_{Z_T}.
\]
Let $U=\Xbar^{{\rm PR}}-Z_{\tau_{\beta,i}}$. Then if $i>1$, we have 
\[
\omega_{\beta,i}|_U \cong \omega_{\beta,i-1}|_U.
\]
For $i=1$, we have 
\[
\omega_{\beta,1}|_U \cong \omega_{\phi^{-1}\beta,e_\gerp}^{p}|_U.
\]
\end{lem}

\begin{proof} 
Assume $i>1$ and $\tau_{\beta,i} \in T$. Consider the map $[\varpi_\gerp]:\HH_{\beta,i} \rightarrow \HH_{\beta,i-1}$ defined in \S \ref{subsection: HMVtame}. It has image $\omega_{\beta,i-1}$. By assumption, $H_{\beta,i}=0$ on $Z_T$, and hence the map 
$m_{\beta,i}: \omega_{\beta,i} \rightarrow \omega_{\beta,i-1}$ vanishes on $Z_T$. By definition, this implies that $[\varpi_\gerp]\omega_\beta^{(i)} \subset \omega_\beta^{(i-2)}$. In particular, $\omega_{\beta,i}$ sits in the kernel of the map $[\varpi_\gerp]$ above. Rank considerations show that we have a short exact sequence of sheaves on $Z_T$:
\[
0 \rightarrow \omega_{\beta,i} \rightarrow \HH_{\beta,i} \rightarrow \omega_{\beta,i-1} \rightarrow 0,
\]
from which it follows $\calO_{Z_T} \cong \wedge^2_{\calO_{Z_T}} \HH_{\beta,i} \cong \omega_{\beta,i}|_{Z_T} \otimes \omega_{\beta,i-1}|_{Z_T}$.

Now assume $\tau_{\beta,1} \in T$. Consider the morphism ${\rm V}_{\beta,1}: \HH_{\beta,1} \rightarrow \omega^{\otimes p}_{\phi^{-1}\beta,e_\gerp}$ defined in \S \ref{section: Hasse}. Since by assumption $H_{\beta,1}=0$ on $Z_T$, we have ${\rm V}_{\beta,1}(\omega_{\beta,1})=0$. Rank considerations imply we have an exact sequence
\[
0 \rightarrow \omega_{\beta,1} \rightarrow \HH_{\beta,1} \rightarrow \omega_{\phi^{-1}\beta,e_\gerp}^{\otimes p} \rightarrow 0.
\]
It follows that $\calO_{Z_T} \cong \wedge^2_{\calO_{Z_T}} \HH_{\beta,1} \cong \omega_{\beta,1}|_{Z_T} \otimes \omega^{\otimes p}_{\phi^{-1}\beta,e_\gerp}|_{Z_T}$.

The other two isomorphisms follow immediately from the definition of Hasse invariants.

\end{proof}
 
 \begin{cor} \label{cor: torsion} Let $T \subset \Sigma$. For any $\tau \in \Sigma$, we have 
\[
\omega_\tau^{\otimes (p^{2f_\gerp}- 1)}|{W_T} \cong \calO_{W_T}.
\]
\end{cor}

\begin{cor}\label{cor: torsion on closed stratum}  Let $\gerp \in \mathbb{S}$. If $\Sigma_\gerp \subset T \subset \Sigma$, then for all $\tau \in \Sigma_\gerp$ we have 
\[
\omega_\tau^{\otimes (p^{2f_\gerp}- 1)}|{Z_T} \cong \calO_{Z_T}.
\]
\end{cor}

\begin{lem}\label{lem: ample} There are integers $m_\tau$ for $\tau \in \Sigma$ such that $ \bigotimes_{\tau} \omega_\tau^{\otimes m_\tau}$ is an ample line bundle on $\Xbar^{{\rm PR}}$.
\end{lem}
\begin{proof}  Let $\delta: \Xbar^{\rm PR} \rightarrow \Xbar^{\rm DP}$ be the forgetful morphism. Let $\omega \subset \HH$ be the Hodge bundle on $\Xbar^{\rm DP}$. Then the determinant $\wedge^d\omega$ is an ample line bundle on $\Xbar^{\rm DP}$ (see, for example, \cite[Theorem 4.3.(ix)]{C}) and we have 
\[
\delta^*(\wedge^d\omega) = \bigotimes_{\gerp \in \mathbb{S}} \bigotimes_{\beta \in \BB_\gerp} \bigotimes^{e_\gerp}_{i=1} \ \ \omega_{\beta,i}.
\]
It is easy to see  that $\calM=\bigotimes_{\gerp \in \mathbb{S}} \bigotimes_{\beta \in \BB_\gerp} \bigotimes^{e_\gerp}_{i=1} \omega^{\otimes (i-e_\gerp-1)}_{\beta,i}$  is relatively ample with respect to $\delta$ (see \cite[Lemma 2.7]{RX}). By \cite[Lemma 0892]{StackProject}, it  follows  that for $a$ large enough, $\calM \otimes \delta^*(\wedge^d\omega)^{\otimes a}=\bigotimes_{\gerp \in \mathbb{S}} \bigotimes_{\beta \in \BB_\gerp} \bigotimes^{e_\gerp}_{i=1} \omega^{\otimes (a+ i-e_\gerp-1)}_{\beta,i}$ is ample on $\Xbar^{\rm PR}$. 
\end{proof}  
   We can now prove the main result of this section.
 \begin{prop}\label{prop: quasi-affine} Let $T \subset \Sigma$. Then $W_T$ is a quasi-affine scheme.
 \end{prop}
  
  \begin{proof} Since $W_T$ is quasi-compact, it is enough to show that there is a line bundle on $W_T$ that is both ample and torsion. Let $m_\tau$ be integers as in Lemma \ref{lem: ample}. Then $\calL=\bigotimes_{\tau} \omega_\tau^{\otimes m_\tau}$ is ample on $W_T$. On the other hand, by Corollary \ref{cor: torsion}, each line bundle $\omega_\tau$, and hence $\calL$, is a torsion line bundle on $W_T$. 
  \end{proof}
  
  \begin{rem} In various works, Oort proves quasi-affineness of strata on Shimura varieties via the existence of a torsion ample line bundle on them. In \cite{GO} this approach is used to prove that open Goren-Oort strata on a Hilbert modular variety are quasi-affine in the case $p$ is unramified in $F$;  to construct a torsion ample line bundle on an open stratum, the authors use the relationship between the Ekedahl-Oort stratification and the type stratification. In \cite{AG2} it is shown that the usual definition of  the Ekedahl-Oort stratification is inadequate by giving an example where there are infinitely many isomorphism classes for the $p$-torsion over $\Xbar^{\rm DP}$. The proof of Proposition \ref{prop: quasi-affine} works directly over the generalized Goren-Oort strata, giving a new proof in the case $p$ is unramified.
  \end{rem}

 \begin{cor}\label{cor: intersect} Let $T \subset \Sigma$. Let $C$ be a (non-empty) irreducible component of $Z_T$. Then for all $\tau \not \in T$, we have $C \cap Z_\tau \neq\emptyset$.
\end{cor} 

\begin{proof} If not, let  $T\subset \Sigma$ be of maximal cardinality where $Z_T$ has a (non-empty) component $C$ such that $C \cap Z_{\tau_0} =\emptyset$ for some $\tau_0\not \in T$ (in particular, $|T|< d$). It follows that $C \cap Z_\tau =\emptyset$ for all $\tau \not\in T$, that is $C \subset W_T$. Since $C$ is quasi-compact, Proposition \ref{prop: quasi-affine} implies that $C$ is quasi-affine. We cannot have  $T=\emptyset$, since otherwise the assumptions would imply that  $\Xbar^{\rm PR}$ (and hence $\Xbar^{\rm DP}$)  has a component composed entirely of ordinary HBAS's, which is false (see, for example, \cite[Cor 8.18]{AG}). In particular, $C$ is both proper and quasi-affine, and hence finite. On the other hand, since $Z_T$ is the intersection of $|T|$ divisors on the equidimensional variety $\Xbar^{\rm PR}$  of dimension $d$, its irreducible components must have dimension at least $d-|T|>0$. This is a contradiction.
\end{proof}

To close this section we summarize some properties of the stratification on $\Xbar^{\rm PR}$. 
 \begin{prop}  \label{prop:GO} The following statements hold.
\begin{enumerate}
\item The collection $\{W_T\}_{T }$ is a stratification of $\Xbar^{\rm PR}$, i.e., for every subset $T\subset \Sigma$, we have $Z_T=\overline{W}_T=\cup_{T' \supseteq T} W_{T'}$.
\item Each stratum $Z_T$ is non-empty and equidimensional of dimension $d-|T|$.
\item Each open stratum $W_T$ is non-empty, quasi-affine, and equidimensional of dimension $d-|T|$.
\end{enumerate}

\end{prop}

\begin{proof} We remark that the non-emptiness of $Z_T$ follows\footnote{See \cite{DDW} for a different proof of non-emptiness of strata.} from Corollary \ref{cor: intersect}. We prove that each $Z_T$ is equidimensional of dimension $d-|T|$. This is true for $T=\emptyset$.  It is also true for $T=\Sigma$, because $Z_\Sigma=W_\Sigma$ is both projective and quasi-affine (by Proposition \ref{prop: quasi-affine}), and hence finite.  Order the elements of $\Sigma=\{\tau_1,...,\tau_d\}$ such that $T=\{\tau_1,...,\tau_r\}$. For $0 \leq i \leq d$, let $T_i=\{\tau_1,...,\tau_i\}$ (in particular, $T_0=\emptyset$). In the chain 
\[
Z_\Sigma=Z_{T_d} \subset Z_{T_{d-1}} \subset ...\subset Z_{T_r}=Z_T \subset ...\subset Z_{T_1}  \subset Z_{T_0}=\Xbar^{\rm PR}
\]
 every inclusion is  of codimension at most one (by Corollary \ref{cor: intersect}). Since $Z_{T_d}$ and $Z_{T_0}$ are equidimensional of dimensions $0$ and $d$, respectively, it follows that $Z_{T_r}=Z_T$ is equidimensional of codimesion $d-r=d-|T|$. Note that equidimensionality also follows from \cite[Theorem 3.10 ]{RX} (See also \cite[Proposition 4.2.1]{D}).

To show that $W_T$ is nonempty for all $T$, note that $W_T=Z_T - \bigcup_{T' \supsetneq T} Z_{T'}$ and that $\dim(Z_{T'}) < \dim(Z_T)$ if $T' \supsetneq T$.  This also shows that $W_T$ is equi-dimensional of dimension $\dim(Z_T)=d-|T|$. 
We now prove that $\{W_T\}_T$ form a stratification of $\Xbar^{\rm PR}$. Since $Z_T$ is closed, we have $\overline{W}_T \subset Z_T$. By dimension considerations, it follows that $\overline{W}_T$ is a union of components of $Z_T$. If $\overline{W}_T \neq Z_T$, then a non-empty union of components of $Z_T$ will have to lie inside $Z_T-W_T=\bigcup_{T' \supsetneq T} Z_{T'}$, which is impossible since every $Z_{T'}$ appearing here has dimension smaller than that of $Z_T-W_T$. The quasi-affineness  of $W_T$ is Proposition \ref{prop: quasi-affine}.

\end{proof}



\section{Pappas-Rapoport models in Iwahori level} \label{subsection: IwahoriHMV}  
Fix $\gerp$ a prime of ${\calO_F}$ above $p$. Recall that $\{\gerJ_1,\cdots,\gerJ_r\}$ is a set of representatives for ${\rm Cl}^+(F)$. For each $1 \leq i \leq r$, there is a unique $1\leq \gamma(i) \leq r$ such that $\gerp\gerJ_i$ and $\gerJ_{\gamma(i)}$ represent the same class in ${\rm Cl}^+(F)$. We fix  once and for all a choice of a totally positive isomorphism $\epsilon_i: \gerJ_{\gamma(i)} \cong \gerp\gerJ_{i}$.

Let $S$ be a locally noetherian $\calO$-scheme  and $\uA/S \equiv (A/S, \iota_A, \lambda_A, \alpha_A, \omega^\bullet_{A/S})$ an $S$-point of $X^{\rm PR}$. Then $A \otimes_{\calO_F} \gerp^{-1}$  can be naturally equipped with a structure of a HBAS $\underline{A \otimes_{\calO_F} \gerp^{-1}}$. We make this into a filtered HBAS by adding a Pappas-Rapoport filtration on $\omega_{(A \otimes_{\calO_F} \gerp^{-1})/S}\cong\omega_{A/S} \otimes_{\calO_F} \gerp$ by applying $ - \otimes_{\calO_F} \gerp$ to the Pappas-Rapoport filtration on $\omega_{A/S}$.

Let $\uA/S = (A/S, \iota_A, \lambda_A, \alpha_A, \omega^\bullet_{A/S})$ and $\uB/S = (B/S, \iota_{B}, \lambda_{B}, \alpha_{B},\omega^\bullet_{B/S})$ correspond, respectively, to $S$-points of $X^{\rm PR}_{\gerJ_i}$ and $X^{\rm PR}_{\gerJ_{\gamma(i)}}$, for some $1 \leq i \leq r$.  An isogeny $f: A \rightarrow B$ is called a cyclic $(\calO_F/\gerp)$-isogeny of filtered  HBAS's if
\begin{enumerate}
 
\item $f$ is an ${\calO_F}$-isogeny of degree $p^{f_{\gerp}}$;

 \item $f$ is compatible with the polarizations, i.e., the  following diagram is commutative
 \[
\xymatrix{ \calP_B \ar[d]_{f^*} \ar[rr]^{\lambda_B} &&\gerJ_{\gamma(i)} \ar[d]^{\bar{\epsilon}_i} \\ \calP_A \ar[rr]^{\lambda_A} && \gerJ_i}
 \]
 where $\bar{\epsilon}_i$ is the composition  $\gerJ_{\gamma(i)} \cong_{\epsilon_i} \gerp\gerJ_i \subset \gerJ_i$, and $f^*(\eta)=f^\vee\eta f$;
 \item $f$ is compatible with the level structures, i.e., $\alpha_B=f\circ \alpha_A$;
 \item $f$ preserves the Pappas-Rapoport filtrations, i.e., for all $\beta \in \BB$,  the induced map $f^*_\beta: \omega_{B/S,\beta} \rightarrow \omega_{A/S,\beta}$ preserves the filtrations $\omega_{A/S,\beta}^\bullet$ and $\omega_{B/S,\beta}^\bullet$. 
 \end{enumerate}

 Let $Y^{{\rm PR}}_{\gerJ_i}/\calO$ be the Hilbert modular scheme classifying the isomorphism classes of tuples $(\uA/S,\uB/S, f:A \rightarrow B, g: B \rightarrow A \otimes_{\calO_F} {\gerp^{-1}})$, where 

\begin{itemize}
\item $S$ is a locally Noetherian $\calO$-scheme;
\item $\uA/S = (A/S, \iota_A, \lambda_A, \alpha_A, \omega^\bullet_{A/S})$ and $\uB/S = (B/S, \iota_{B}, \lambda_{B}, \alpha_{B},\omega^\bullet_{B/S})$ correspond to $S$-points of $X^{\rm PR}_{\gerJ_i}$ and $X^{\rm PR}_{\gerJ_{\gamma(i)}}$, respectively;
\item $f$ and $g$ are cyclic $(\calO_F/\gerp)$-isogenies of filtered HBAS's;
\item The isogeny $g \circ f: A  \ra A \otimes_{{\calO_F}} \gerp^{-1}$ is the natural map induced by the inclusion of $\calO_F$ in $\gerp^{-1}$. (Equivalently, the isogeny $f \circ (g\otimes_{{\calO_F}} \gerp): B \otimes_{\calO_F} \gerp \ra B$ is the natural map induced by the inclusion of $\gerp$ in ${\calO_F}$.)
 \end{itemize}

The Hilbert modular variety $Y^{\rm PR}_{\gerJ_i}$ is independent of the choice of $\epsilon_i$ up to isomorphism. We define   
\[
Y^{\rm PR}=\bigcup_{i=1}^r\  Y^{{\rm PR}}_{\gerJ_i}.
\]

Let $\Ybar^{\rm PR}=Y^{\rm PR} \otimes_{\calO} \FF$ and denote by $\pi_1$ and $\pi_2$ the forgetful  maps
\[
\pi_{1}:\Ybar^{\rm PR} \rightarrow \Xbar^{\rm PR}, \qquad \pi_{2}:\Ybar^{\rm PR} \rightarrow \Xbar^{\rm PR} 
\]
where $\pi_1(\uA,\uB,f,g)=\uA$, and $\pi_2(\uA,\uB,f,g)=\uB$. 


We recall a stratification defined on $\Ybar^{\rm PR}$ in  \cite{GK} when $p$ is unramified in $F$, and in \cite{ERX} in the general case. (See also \cite{H} for a similar construction in a related unitary case).  Let  $(\uA,\uB,f,g)$ be the universal tuple defined over $\Ybar^{\rm PR}$. Since for all $\beta \in \BB$, the pullback morphisms $f_\beta^*: \omega_{B/S,\beta} \ra \omega_{A/S,\beta}$ and $g_\beta^*: \omega_{A\otimes_{\calO_F}\gerp^{-1}/S,\beta} \ra \omega_{B/S,\beta}$ preserve
Pappas--Rapoport filtrations, we have induced morphisms
\[
f_{\beta,i}^*  \colon \omega_{B/S,\beta,i} \rightarrow \omega_{A/S,\beta,i},
\]
 \[
g_{\beta,i}^*  \colon \omega_{(A\otimes_{\calO_F}\gerp^{-1})/S,\beta,i} \rightarrow \omega_{B/S,\beta,i},
\]
where for each $\beta \in \BB_\gerp$, $i$ runs between $1$ and $e_\gerp$. For $\tau = \tau_{\beta,i} \in \Sigma_{\gerp}$, let $U_{\tau}$ be the vanishing locus of $f_{\beta,i}^*$ on $\Ybar^{\rm PR}$, and  $V_{\tau}$  the vanishing locus of $g_{\beta,i}^{*}$ on $\Ybar^{\rm PR}$. For $\varphi,\eta$ subsets of $\Sigma_{\gerp}$ satisfying $\sigma^{-1}\varphi \cup \eta =\Sigma_{\gerp},$ we define
\[
Z_{\varphi,\eta}=\bigcap_{\tau \in \varphi} U_{\sigma^{-1}\tau} \cap \bigcap_{\tau \in \eta} V_{\tau}.
\]


Notice the shift in $\varphi$ in our notation versus that in \cite{ERX}, which is consistent with the notation in \cite{GK}. One can show that the collection of $Z_{\varphi,\eta}$ defined above form a stratification of $\Ybar^{\rm PR}$. The local and global properties of the strata are studied in detail in \cite{GK} and \cite{ERX}. In the following, we recall some facts about this stratification that we will use in our argument.

Let $r$ be a non-negative integer. A morphism of $\FF$-schemes $\lambda\colon Z_1 \rightarrow Z_2$ is called an $r$-Frobenius factor if there is a morphism  $\lambda^\prime\colon Z_2 \rightarrow Z_1$ such that $\lambda\lambda^\prime\colon Z_2 \rightarrow Z_2$ is the $r$-th power of the absolute Frobenius morphism of $Z_2$. By \cite[Proposition 4.8]{H}, if $Z_1,Z_2$ are schemes of finite type over $\FF$ such that $Z_2$ is normal, then any proper morphism $\lambda\colon Z_1 \rightarrow Z_2$ which is a bijection on points is an $r$-Frobenius factor for some $r \geq 0$.

%

Fix $\tau_0 \in \Sigma_{\gerp}$. Recall the definition of $\HH_{\tau_0}$ in (\ref{eqn: H_tau}). We let
\[
{\rm pr}: \PP \rightarrow Z_{\sigma\tau_0}
\]
be the $\PP^1$ bundle given on $Z_{\sigma\tau_0} \subset X^{\rm PR}$ by  $\PP(\HH_{\tau_0})$. Let $\{\tau_0\}^c=\Sigma_{\gerp}-\{\tau_0\}$.

 \begin{prop} \label{prop: frob factor} If $|\Sigma_{\gerp}| > 1$, then $\pi_1(Z_{\{\sigma\tau_0\},\{\tau_0\}^c})=Z_{\sigma\tau_0}$, and $\pi_2(Z_{\{\sigma\tau_0\},\{\tau_0\}^c})=Z_{\tau_0}$.  Furthermore, there is an integer $r \ge 0$,  and  an $r$-Frobenius factor $\lambda$ which makes the following diagram commutative:
\begin{equation}\label{Diagram: frob}\xymatrix{Z_{\{\sigma\tau_0\},\{\tau_0\}^c}\ar[drr]_{\pi_1} \ar[rr]^{\lambda} &&
\PP\ar[d]^{\rm pr} \\ && Z_{\sigma\tau_0} .}\end{equation}
 \end{prop}

 \begin{proof} See \cite[Thm 2.6.4. (2)]{GK} and \cite[Prop. 4.5.]{ERX}, noting the difference in notation between ours and \cite{ERX}: in general, the stratum $Z_{\varphi,\eta}$ is denoted by $Y_{\sigma^{-1}\varphi,\eta} $ in {\it loc. cit}.  For the final statement, note that from our definitions we have $\pi_2(Z_{\{\sigma\tau_0\},\{\tau_0\}^c})=\pi_1(Z_{\{\sigma\tau_0\}^c,\{\tau_0\}})=Z_{\tau_0}$.
 
 \end{proof}

For $\tau=\tau_{\beta,i} \in \Sigma_{\gerp}$, we define $n_\tau=p$ if $i=1$, and $n_\tau=1$ otherwise. By the proposition, there exists $\lambda': \PP \rightarrow Z_{\{\sigma\tau_0\},\{\tau_0\}^c}$ such that $\lambda\lambda^\prime:\PP \rightarrow \PP$ is the $r$-th power of the absolute Frobenius morphism. We will be considering the following composite morphism in the proof of our main theorem:
\begin{equation*}
 \xymatrix{\PP \ar[rr]^{\lambda'} && Z_{\{\sigma\tau_0\},\{\tau_0\}^c}\ar[rr]^{\pi_2} &&Z_{\tau_0}.}
\end{equation*}

\begin{lemma}\label{lemma: fibres} Let $\tau_0\in \Sigma_{\gerp}$ and assume that $f_{\gerp} > 1$.   Let $x$ be a geometric point of $Z_{\sigma\tau_0}$, and denote by $\PP^1_x$ the fibre at $x$ of of the $\PP^1$-bundle $\PP$. Then  $n_{\tau_0}|p^r$ and  for $\tau \in \Sigma_{\gerp}$ we have:
 $$(\pi_2\lambda')^*\omega_\tau|_{\PP^1_x} \cong \begin{cases} 
   \calO_{\PP^1_x}(p^r)& \mbox{if $\tau=\tau_0$,} \\
  \calO_{\PP^1_x}(-p^{r}/n_{\tau_0}) & \mbox{if $\tau=\sigma^{-1}\tau_0$,}\\
 \calO_{\PP^1_x} & \mbox{otherwise.}
  \end{cases}$$
\end{lemma}

\begin{proof} Consider the following diagram

\begin{equation}\label{Diagram: fibres}\xymatrix{\PP\ar[rr]^{\lambda'}\ar[d]_{\pi_2\lambda'}  && Z_{\{\sigma\tau_0\},\{\tau_0\}^c}\ar[drr]_{\pi_1} \ar[rr]^{\lambda}\ar[dll]_{\pi_2} &&
\PP\ar[d]^{\rm pr} \\ Z_{\tau_0}&&&& Z_{\sigma\tau_0} .}\end{equation}
The following statements are consequences of \cite[Prop. 4.7]{ERX} (note that $Z_{\{\sigma\tau_0\},\{\tau_0\}^c}$ is denoted $Y_{\{\tau_0\},\{\tau_0\}^c}$ in {\it loc. cit}).
\begin{enumerate} 
\item $\pi_2^\ast\omega_{\tau_0}=\lambda^\ast\calO_{\PP}(1)$;
\item $\pi_2^\ast \omega^{\otimes n_{\tau_0}}_{\sigma^{-1}\tau_0}=\lambda^\ast\calO_{\PP}(-1)$;
\item $\pi_2^\ast\omega^{\otimes n_{\sigma\tau}}_\tau=\pi_1^\ast\omega_{\sigma\tau}$, if $\tau \not \in \{\tau_0,\sigma^{-1}\tau_0\}$.
\end{enumerate}
 From (1) it follows that $(\pi_2\lambda')^\ast \omega_{\tau_0}=\lambda'^\ast\lambda^\ast\calO_{\PP}(1)=\calO_{\PP}(p^r)$. From (2) it follows that $(\pi_2\lambda')^\ast \omega^{\otimes n_{\tau_0}}_{\sigma^{-1}\tau_0}=\calO_{\PP}(-p^r)$. In particular, this  implies that $n_{\tau_0}|p^r$, and hence we find  $(\pi_2\lambda')^\ast \omega_{\sigma^{-1}\tau_0}|_{\PP^1_x} =\calO_{\PP^1_x}(-p^r/n_{\tau_0})$. Finally, (3) implies that for $\tau \neq \tau_0, \sigma^{-1}\tau_0$ we have 
$(\pi_2\lambda')^\ast\omega^{\otimes n_{\sigma\tau}}_\tau=(\lambda')^\ast\pi_1^\ast\omega_{\sigma\tau}=(\lambda\lambda')^\ast\pr^\ast\omega_{\sigma\tau}$. Since $\pr^\ast\omega_{\sigma\tau}$ restricted to every fibre of $\pr$ is the trivial line bundle, it follows that $(\pi_2\lambda')^\ast\omega^{\otimes n_{\sigma\tau}}_\tau$ restricted to $\PP^1_x$ is trivial, and hence so is the restriction of $(\pi_2\lambda')^\ast\omega_{\sigma\tau}$ to each fibre $\PP^1_x$.

\end{proof}

\section{The Main Theorem}

In this section, we prove our main theorem and discuss some of its consequences. Given what we have proven in the previous sections, the rest of the proof is in essence similar to the case where $p$ is unramified in $\calO_F$, treated in \cite{DK}. We would like to emphasize  again that in this more general proof we bypass the use of deep geometric results on global properties of the Goren-Oort strata on $\Xbar^{\rm PR}$ (currently unavailable in the case $p$ is ramified in $F$), and instead appeal to properties of a stratification on $\Ybar^{\rm PR}$ that are more readily generalizable to other Shimura varieties.  

Recall that for $\tau \in \Sigma$, we have $H_\tau \in M_{{\bf h}_\tau}(\gn;\FF)$, with ${\bf h}_\tau=n_\tau{\bf e}_{\sigma^{-1}\tau}-{\bf e}_\tau$, where $n_\tau=p$ if $i=1$, and $n_\tau=1$ otherwise.

\begin{thm} \label{thm:key}  Let $\gerp |p$ be a prime of $\calO_F$.  Let  ${\bf k}=\! \sum k_\tau{\bf e}_\tau \in \ZZ^{\Sigma}$ be a weight. Suppose that we have 
\[
n_{\tau_0}k_{\tau_0} < k_{\sigma^{-1}{\tau_0},}
\]
for some $\tau_0\in \Sigma_{\gerp} \subset \Sigma$.  Then, multiplication by the
generalized partial Hasse invariant $H_{\tau_0}$ induces an isomorphism:
$$M_{{\bf k}-{\bf h}_{\tau_0}}(\gn;\Fpbar) \,\,\,\stackrel{\sim}{\lra} \,\,\, M_{\bf k}(\gn; \Fpbar).$$
When $\Sigma_{\gerp}= \{\tau_0\}$ (so that the above condition is equivalent to $k_{\tau_0}<0$),  a stronger result holds:  if $k_{\tau_0}<0$, then $M_{\bf k}(\gn; \Fpbar)=0$.
\end{thm}

\begin{proof}   First we assume $|\Sigma_{\gerp}|>1$.  We have a short exact sequence 
\[
0 \rightarrow \omega^{{\bf k}-{\bf h}_{\tau_0}} \rightarrow \omega^{\bf k} \rightarrow \calF \rightarrow 0
\]
of sheaves on $\Xbar^{\rm PR}$, where the first map is  multiplication by $H_{\tau_0}$. The quotient sheaf $\calF$ is supported on  $Z_{\tau_0}$, and we have $H^0(\Xbar^{\rm PR},\calF)=H^0(Z_{\tau_0}, \omega^{\bf k})$. It is hence  enough to show that  $H^0(Z_{\tau_0}, \omega^{\bf k})=0$ under the assumptions. Consider the morphism
\[
\pi_2\lambda'\colon \PP \rightarrow Z_{\tau_0}
\]
(see Proposition \ref{prop: frob factor}).  Since $Z_{\tau_0}$ is reduced and $\pi_2\lambda'$ is surjective, we have an injection \[
(\pi_2\lambda')^*: H^0(Z_{\tau_0},\omega^{\bf k}) \injects H^0(\PP, (\pi_2\lambda')^*\omega^{\bf k}),
\]
It is therefore enough to show that $H^0(\PP, (\pi_2\lambda')^*\omega^{\bf k})=0$. We prove that for any  $s\in H^0(\PP, (\pi_2\lambda')^*\omega^{\bf k})$, the restriction of $s$ to every geometric fibre of  the $\PP^1$-bundle ${\rm pr}: \PP \rightarrow Z_{\sigma\tau_0}$ is zero. Let $x$ be a geometric point of $Z_{\sigma\tau_0}$, and $\PP^1_x$ the fibre over $x$. By Lemma \ref{lemma: fibres}, we have 
\[
(\pi_2\lambda')^*\omega^{\bf k}|_{\PP^1_x}\cong \calO_{\PP^1_x}(p^rk_{\tau_0}-(p^{r}/n_{\tau_0})k_{\sigma^{-1}\tau_0}).
\]
By assumption $p^rk_{\tau_0}-(p^{r}/n_{\tau_0})k_{\sigma^{-1}\tau_0}<0$, which implies that $s|_{\PP^1_x}=0$ as desired.

Now, assume $\Sigma_{\gerp}= \{\tau_0\}$. We prove that if $k_{\tau_0}<0$, then $H^0(\Xbar^{\rm PR}, \omega^{\bf k})=0$.  The proof is essentially the same as in the unramified case given in \cite[Theorem 5.1]{DK}, but we include it here to highlight where Corollary \ref{cor: intersect} is used.
Let $\Sigma=\{\tau_0,\tau_1,...,\tau_{d-1}\}$. Let  $T_i:=\{\tau_{i+1},...,\tau_{d-1}\}$ for $0 \leq i \leq d-1$ (where $T_{d-1}=\emptyset$).  We prove by induction that $H^0(Z_{T_i},\omega^{\bf k})=0$ for all $0 \leq i \leq d-1$. Since $Z_{T_{d-1}}=\Xbar^{\rm PR}$, this proves the claim. 

We first prove the result for $i=0$.  Let $C$ be a component of $Z_{T_{0}}$. By \cite[Thm. 3.10]{RX}, $C$ is a smooth projective curve. By Corollary \ref{cor: intersect}, $H_{\tau_0}\in H^0(\Xbar^{\rm PR},\omega_{\tau_0}^{p-1})$ vanishes at some point on $C$. It follows that  $\omega_{\tau_0}$ has positive degree on $C$. By  Corollary \ref{cor: torsion on closed stratum},  every $\omega_{\tau}$ for $\tau \neq \tau_0$ restricts to a torsion line bundle on $C$ (since $f_{\gerp}$=1, and consequently $\Sigma_\gerq \subset T_0$ for all $\gerq \neq \gerp$). It follows that  every  $\omega_\tau$ for $\tau \neq \tau_0$ has degree zero on $C$. This shows that  $\deg(\omega^{\bf k}|_{C})=k_{\tau_0}\deg(\omega_{\tau_0}|_{C})<0,$ from which it follows  that $H^0(C,\omega^{\bf k})=0$. \footnote{Note that, by \cite[Lecture 11, (VI.)]{M}, for this argument we only need that $C$ is Cohen-Macaulay, which is immediate from Part (2) of Proposition \ref{prop:GO}.}
 
 Now assume that $1 \le i \leq d-1$, and we have $H^0(Z_{T_{i-1}},\omega^{\bf k})=0$ for all ${\bf k}$ with $k_{\tau_0}<0$. We prove that $H^0(Z_{T_i},\omega^{\bf k})=0$ for all such ${\bf k}$ as well. Consider the short exact sequence of sheaves on $Z_{T_i}$
 \[
 0 \ra \omega^{{\bf k}-{\bf h}_{\tau_{i}}}|_{Z_{T_{i}}}\ra \omega^{\bf k}|_{Z_{T_{i}}} \ra \calF_i \ra 0,
 \]
where the first map is multiplication by $H_{\tau_{i}}$. The sheaf $\calF_i$ is supported on the  $Z_{T_{i-1}}$, and we have $H^0(Z_{T_{i}},\calF_i)=H^0(Z_{T_{i-1}}, \omega^{\bf k})=0$. This shows that multiplication by $H_{\tau_{i}}$ induces an isomorphism
  \[
 H^0(Z_{T_{i}}, \omega^{{\bf k}-{\bf h}_{\tau_{i}}}) \,\,\,\stackrel{\sim}{\lra} \,\,\, H^0({Z_{T_{i}}},\omega^{\bf k}).
 \]
The same result holds after replacing ${\bf k}$ with ${\bf k}
^\prime={\bf k}-M {\bf h}_{\tau_{i}}$ for any $M>0$, as $k^\prime_{\tau_0}=k_{\tau_0}<0$. It follows that for all $M>0$, multiplication by  $(H_{\tau_i})^M|_{Z_{T_{i}}}$ induces an isomorphism
 \[
 H^0(Z_{T_{i}}, \omega^{{\bf k}-M{\bf h}_{\tau_i}}) \,\,\,\stackrel{\sim}{\lra} \,\,\, H^0({Z_{T_{i}}},\omega^{\bf k}).
 \]
 By Corollary \ref{cor: intersect}, $H_{\tau_i}$ has a nonempty vanishing locus on every component of $Z_{T_{i}}$. Lemma \ref{lem: maximal power} implies that any nonzero section of $\omega^{\bf k}$ on $Z_{T_{i}}$ is divisible by at most a finite power of $H_{\tau_i}|_{Z_{T_{i}}}$. It follows that $H^0({Z_{T_{i}}},\omega^{\bf k})=0$. This proves the claim by induction. In particular, for $i=d-1$,  we get 
 \[
 H^0(\Xbar^{\rm PR},\omega^{\bf k})=H^0(Z_{T_{d-1}},\omega^{\bf k})=0.
 \]

\end{proof}

We record the following elementary lemma used in the proof of the main theorem, and again in the definition of filtration of Hilbert modular forms in \S \ref{sec:filt}.
\begin{lem}\label{lem: maximal power} Let $X$ be a noetherian scheme. Let $\calL$, $\calM$ be line bundles on $X$. Let $h \in H^0(X,\calL)$ and $f\in H^0(X,\calM)$ be nonzero sections. Assume that the divisor of $h$ intersects every irreducible component of $X$. Then there is a maximal $m\geq 0$ such that $f=h^mg$ for some 
$g \in H^0(X,\calM\otimes \calL^{-m})$.
\end{lem}

  \section{Filtration of mod $p$ Hilbert modular forms.}\label{sec:filt}  In this section, we state two corollaries of our main theorem which answer questions Andreatta and Goren asked in \cite{AG} regarding filtrations of Hilbert modular forms. We have previously answered these questions in the case $p$ is unramified in $\calO_F$ in \cite{DK}.

  Recall that in \cite[\S 8.19]{AG},  Andreatta and Goren define the filtration  of a nonzero mod $p$ Hilbert modular form $f$  as the minimal weight of a modular form from which $f$ arises by multiplication by the partial Hasse invariants $H_\beta$ as $\beta$ runs in $\BB$.  As we have seen, over $\Xbar^{\rm PR}$, the partial Hasse invariants $H_\beta$ can be further factorized into a product of generalized partial Hasse invariants. Here we give a  definition of filtration by considering all generalized partial Hasse invariants; see also \cite[\S 1.4]{DDW}.  This notion coincides with that given by Andreatta-Goren in the case  $p$ is unramified in $F$.

Let $0 \neq f \in M_{\bf k}(\gn, \Fpbar)$. For each $\tau \in \Sigma$, we let $a_\tau$ be the maximal power of $H_\tau$ that divides $f$ (which exists by Corollary \ref{cor: intersect} and Lemma \ref{lem: maximal power}). We define the filtration of $f$ to be
\[
\Phi(f)={\bf k} - \sum_\tau a_\tau {\bf h}_\tau.
\]
Since $\Xbar^{\rm PR}$ is smooth, and since for $\tau_1 \neq \tau_2$ the divisors of $H_{\tau_1}$ and $H_{\tau_2}$ have no common components (by part (2) of Proposition \ref{prop:GO}), we find  that $f$ is divisible by $\prod_\tau (H_\tau)^{a_\tau}$. It follows that $\sum_\tau a_\tau {\bf e}_\tau$ is the unique maximal element of the set
$$\left\{\left.\, \sum_\tau x_\tau {\bf e}_\tau \in \ZZ_{\ge 0}^{\Sigma}\, \,\right|\, \,  f = g\prod_\tau (H_\tau)^{x_\tau}\,\mbox{for some $g \in M_{{\bf k} - \sum_\tau x_\tau {\bf h}_\tau}(\gn, \Fpbar)$}\,\right\}$$
under the usual partial ordering.  In particular, $\Phi(f)$ is the weight of $g$ corresponding to the above maximal element. 

We next define the {\em minimal}, {\em standard}, and {\em Hasse} cones.
 $$C^{\rm min}  = \left\{\left.\, \sum_\tau x_\tau {\bf e}_\tau \in \QQ^{\Sigma}\, \,\right|\, \, n_\tau x_\tau \ge x_{\sigma^{-1}\tau}\,\ \mbox{for all $\tau \in \Sigma$}\,\right\},$$
 $$C^{\rm st}  = \left\{\left.\, \sum_\tau x_\tau {\bf e}_\tau \in \QQ^{\Sigma}\, \,\right|\, \,  x_\tau \ge 0\,\ \mbox{for all $\tau \in \Sigma$}\,\right\},$$
$$C^{\rm Hasse}  = \left\{\left.\, \sum_\tau y_\tau {\bf h}_\tau \in \QQ^{\Sigma} \,\,\right|\, \, y_\tau \ge 0 \,\ \mbox{for all $\tau\in \Sigma$}\,\right\}.$$
It is easy to see that $C^{\rm min}\subset C^{\rm st} \subset C^{\rm Hasse}$,  and equality holds if and only if $p$ is totally split in $F$.
 
%
%
%
 
\begin{corollary}  \label{corollary: filtration}  If $0 \neq f \in M_{\bf k}(\gn,\Fpbar)$ for some weight 
${\bf k} \in \ZZ^{\Sigma}$, then $\Phi(f) \in C^{\rm min}$.
\end{corollary}

\begin{proof}  
Let $\Phi(f)=\sum_\tau w_\tau {\bf e}_\tau$. Write $f = g\prod_\tau {(H_\tau)}^{a_\tau}$, where the integers $a_\tau$ are non-negative, and $g\in M(\gn,\Fpbar)$ is not divisible by any generalized partial Hasse invariant.  
 If $\Phi(f) \not\in C^{\rm min}$, then for some $\tau \in \Sigma$ we have $n_\tau w_\tau < w_{\sigma^{-1}\tau}$. It then follows by Theorem \ref{thm:key} that $g$ is divisible by $H_\tau$, which is a contradiction.
\end{proof}
%
%

\begin{corollary} \label{corollary: vanishing} If $M_{\bf k}(\gn,\Fpbar) \neq 0$, then ${\bf k} \in C^{\rm Hasse}$.
\end{corollary}

\begin{proof} Assume $0 \neq f \in M_{\bf k}(\gn,\Fpbar)$. By Corollary \ref{corollary: filtration}, $\Phi(f)\in C^{\min} \subset C^{\rm Hasse}$. By definition of filtration, there are integers $a_\tau\geq 0$ such that  
\[
{\bf k}=\Phi(f)+\sum_{\tau} a_\tau {\bf h}_\tau \in C^{\rm Hasse}.
\]

\end{proof}

 \end{document}